\documentclass[11pt]{amsart}
\usepackage{amsmath,amssymb,amsfonts,amsthm, amscd,indentfirst}
\usepackage{latexsym,
amscd,amsxtra}
\usepackage[colorlinks=true,urlcolor=blue, citecolor=red,linkcolor=blue,
linktocpage,pdfpagelabels, bookmarksnumbered,bookmarksopen]{hyperref}
\usepackage[hyperpageref]{backref}
\usepackage{color}         
\usepackage{cite}           
\usepackage{pdfpages}
\usepackage{graphics}
\usepackage{enumitem}    
\usepackage[T1]{fontenc}                   
\usepackage[utf8]{inputenc}                 
\usepackage[english]{babel}       
\usepackage{graphicx}                      %
\usepackage[font=small]{quoting}            %
\usepackage{caption}                  %
\usepackage[top=.8in,bottom=1.2in,left=1in,right=1in]{geometry}
\usepackage{verbatim}
\usepackage{color}
\usepackage{picture}
\usepackage{amsmath,amsthm,amsfonts,amssymb} 
\usepackage{comment}

\newcommand{\sym}[1]{\ensuremath{u^*_{#1}}}
\newcommand{\Omegasym}[1]{\ensuremath{\Omega^*_{#1}}}

\newtheorem{thm}{Theorem}[section]

\newtheorem{cor}[thm]{Corollary}
\newtheorem{prop}[thm]{Proposition}

\newtheorem{defn}{Definition}
\theoremstyle{definition}
\newtheorem{rem}[thm]{Remark}
\theoremstyle{remark}

\def\O{\Omega}

\def\S{\Sigma} 
\def\n{\nabla}

\def\p{\partial}

\def\a{\alpha}
\def\b{\beta}
\def\n{\nabla}

\def\o{\omega}
\def\O{\Omega}
\def\p{\partial}

\def\a{\alpha}
\def\b{\beta}
\def\g{\gamma}
\def\d{\delta}
\def\k{\kappa}
\def\l{\lambda}
\def\s{\sigma}

\def\De{\Delta}
\def\n{\nabla}
\def\<{\langle}
\def\>{\rangle}
\def\div{{\rm div}}
\def\De{\Delta}
\def\n{\nabla}

\def\RR{\mathbb{R}}

\def\o{\omega}
\def\O{\Omega}
\def\p{\partial}

\def\a{\alpha}
\def\b{\beta}
\def\g{\gamma}
\def\d{\delta}
\def\l{\lambda}
\def\s{\sigma}

\def\Rn{\mathbb R^n}

\def\de{\partial}

\def\R{\mathbb{R}}

\def\rr{\mathbb{R}}
\def\ve{\varepsilon}

    {\left\{\begin{array}{@{}l@{}}}{\end{array}\right.}
\patchcmd{\abstract}{\scshape\abstractname}{\textbf{\abstractname}}{}{}
\makeatletter      
\def\@makefnmark{} 
\makeatother       

\begin{document}
\title{Symmetrization with respect to mixed volumes}
 \author{Francesco Della Pietra}%
       \address{Universit\`a degli studi di Napoli Federico II, Dipartimento di Matematica e Applicazioni ``R. Caccioppoli'', Via Cintia, Monte S. Angelo - 80126 Napoli, Italia}
\email{f.dellapietra@unina.it}
 \author{Nunzia Gavitone}%
       \address{Universit\`a degli studi di Napoli Federico II, Dipartimento di Matematica e Applicazioni ``R. Caccioppoli'', Via Cintia, Monte S. Angelo - 80126 Napoli, Italia}
      \email{nunzia.gavitone@unina.it}
\author{Chao Xia}
\address{School of Mathematical Sciences\\
Xiamen University\\
361005, Xiamen, P.R. China}
\email{chaoxia@xmu.edu.cn}
\thanks{
\indent F.D.P. and N.G. are supported by GNAMPA of INdAM; C.X. is supported by the  NSFC (Grant No. 11871406) and the Natural Science Foundation of Fujian Province of China (Grant No. 2017J06003)}

\numberwithin{equation}{section}

\begin{abstract}
In this paper we introduce a new symmetrization with respect to mixed volume or anisotropic curvature integral, which generalizes the one with respect to quermassintegral due to Talenti \cite{ta2} and Tso \cite{Tso}. We show a P\'olya-Szeg\H o type principle for such symmetrization --- it diminishes the anisotropic Hessian integral for quasi-convex functions. We achieve this by a systematic study of invariants on non-symmetric matrices with real eigenvalues and the higher order anisotropic mean curvatures of level sets, which may be of independent interest. As applications, we establish a comparison principle for anisotropic Hessian equations and sharp anisotropic Sobolev inequalities. \\

\noindent {\bf MSC 2010:} 35A23, 52A39, 35J25. \\
{\bf Keywords:}  Symmetrization, Mixed volumes, Anisotropic curvature, Hessian equation, P\'olya-Szeg\H o principle. \\
\end{abstract}

\date{}
\maketitle
\begin{center}
\begin{minipage}{13cm}
\small
\tableofcontents
\end{minipage}
\end{center}

\

\section{Introduction}

The Schwarz symmetrization, introduced by Schwarz \cite{Sch} in 1884, is a symmetrization process which assigns to a given function, a radially symmetric function whose super or sub level sets have the same volume as that of the given function. One of the most important properties of the Schwarz symmetrization is that it diminishes the Dirichlet integral. This is the so-called P\'olya-Szeg\H o principle.
The Schwarz symmetrization plays an important role in proving sharp geometric or analytic inequalities in analysis, geometry and mathematical physics, see e.g. \cite{Bandle, polyaszego}.

Talenti \cite{ta2} introduced, in the planar case, a new symmetrization which preserves the perimeter of the level sets. More precisely, this assigns to a given convex function, a radially symmetric function whose super or sub level sets have the same perimeter as that of the given function. Talenti showed that this new symmetrization diminishes the Monge-Amp\`ere functional in two dimensions. Later, Tso \cite{Tso} generalized Talenti's symmetrization to any dimensions.  Moreover, he introduced similar symmetrizations with respect to quermassintegrals or curvature integrals, which turn out to diminish the Hessian integrals.

To be precise, for an integer $1\le k\le n$, given a bounded convex domain $\O  \subset\R^{n}$ and a strictly convex function $u$ with $u|_{\p \O}=0$, Talenti-Tso's $(k-1)$-th symmetrand $u^{\ast}_{k-1}$ of $u$ is defined by
\begin{equation*}
\label{tt-s}
u^{\ast}_{k-1}(x)=\sup\left\{t\leq 0: W_{k-1}(\{u\le t\})\le \o_n |x|^{n+1-k}\right\} 
\end{equation*}
where $W_{k-1}$ is the $(k-1)$-quermassintegral functional and $\o_n$ is the volume of the unit ball.
The case $k=1$ reduces to the classical Schwarz symmetrization.
The Hessian integral is defined by 
\begin{equation}
\label{hess-integ}
I_{k}[u,\Omega]=\int_{\Omega}(-u) S_{k}(\n^2 u)\,dx,
\end{equation}
where $(\n^2 u)$ is the Hessian matrix of $u$ and $S_k$ denotes the $k$-th elementary symmetric function on the eigenvalues of $(\n^2 u)$.
Talenti \cite{ta2} and Tso \cite{Tso} showed the following P\'olya-Szeg\H o type principle for Hessian integrals under Talenti-Tso's symmetrization: for a strictly convex function $u$ on a convex domain $\O$,
\begin{equation*}
\label{psi}
I_{k}[u,\Omega]\ge I_{k}[u^{\ast}_{k-1}, \Omega^{\ast}_{k-1}],
\end{equation*}
where $\Omega^{\ast}_{k-1}$ is a ball having the same $(k-1)$-th quermassintegral as $\O$. 

As a direct application of Talenti-Tso's symmetrization, a priori estimates for solutions of the Dirichlet problem for Hessian equation can be derived, see \cite{ta2, Tso} for details. See also Trudinger \cite{tr2}, which generalized Talenti-Tso's symmetrization to functions with some weak convexity condition on non-convex domains. For many other interesting properties and more general results on such kind of symmetrizations, we refer also the reader, for example, to \cite{bntpoin,bt07,bt07bis,dpg2,dpg3,ga09,sa12} and to the references therein.

In another direction, Alvino et.al.\cite{aflt} introduced a generalized Schwarz symmetrization, which they called convex symmetrization, so that the super or sub level sets of the resulting function after this symmetrization are Wulff balls, with respect to a suitable norm $F$. Analog that the Schwarz symmetrization decreases the Dirichlet integral, the convex symmetrization diminishes the anisotropic Dirichlet integral with respect to $F$.
Motivated by \cite{aflt} and \cite{ta2}, the first two authors studied in \cite{dpgmaan} a symmetrization with respect to the anisotropic perimeter in two dimensions and they proved that this symmetrization decreases the anisotropic Monge-Amp\`ere functional. 
In view of Tso's result \cite{Tso}, one may naturally ask whether this result holds for any dimensions, and more generally, whether a similar symmetrization with respect to the anisotropic curvature integral diminishes a corresponding anisotropic Hessian integral. This is the main objective of this paper.

Let $F\in C^{3}(\R^{n}\setminus\{0\})$ be a strongly convex norm on $\R^n$. Let $\O\subset \rr^n$ be a bounded convex domain with $C^{2}$ boundary and $u\in C^2(\O)$ be a quasi-convex function with $u|_{\p \O}=0$. Here a quasi-convex function means a function whose {sub}level sets are convex.
The anisotropic $(k-1)$-th symmetrand of $u$ with respect to $F$ is defined by
\begin{equation*}
	\sym {k-1,F}(x)=\sup\left\{t: W_{k-1,F}(\{u\le t\})\le \k_n \left(F^{o}(x)\right)^{n+1-k}\right\},
\end{equation*}
where $W_{k-1,F}$ is the $(k-1)$-th mixed volumes or anisotropic curvature integral with respect to $F$ and $\k_n=|\mathcal W|$ is the volume of the unit Wulff ball and $F^o$ be the dual norm of $F$. 

Mixed volume is a fundamental and important concept in the theory of convex bodies. It appears as coefficients in the Steiner formula which is an expansion of  volume functional on Minkowski's sum for two convex bodies, see \cite{Sch}. The classical Alexandrov-Fenchel inequality in the theory of convex bodies implies an isoperimetric type inequality between mixed volumes, which will be a crucial ingredient to study this symmetrization. For details we refer to Section 3.1 below.

The anisotropic $k$-Hessian integral with respect to $F$ is given by
\[
I_{k,F}[u,\Omega]=\int_{\Omega} (-u)S_{k,F}[u] dx,
\]
where $S_{k,F}[u]$ is the anisotropic $k$-Hessian operator defined as 
\[
S_{k,F}[u]=S_{k}(A_{F}[u]),
\]
where $(A_{F}[u])$ is the anisotropic Hessian matrix given by
\[
{A_{F}}_{ij}[u]=\de_{x_j}\left[ \de_{\xi_{i}}\left(\frac{1}{2}F^{2}\right)\left(\nabla u\right) \right].
\]

When $k=1$, $S_{1,F}$ reduces to the so-called Finsler-Laplacian or anisotropic Laplacian. In recent decades, there are intensive studies on this operator and the corresponding partial differential equations, see for example \cite{BC, CS, CFV, dgmana, WX, WX2}.

The present paper seems to be a first attempt to study  the anisotropic $k$-Hessian operator $S_{k,F}$ for general $k$. We remark that Cianchi-Salani \cite{CS} has proved some basic properties for $S_{2, F}$, in order to study overdetermined problem for $S_{1,F}$, the Finsler-Laplacian.

The main result of this paper is the following P\'olya-Szeg\H o type principle for the anisotropic Hessian integral (see Theorem \ref{pol})
\[
I_{k,F}[u,\Omega]\ge I_{k,F}[u^{*}_{k-1,F},\Omega^{*}_{k-1,F}].
\]


On one hand, when $F$ is the Euclidean norm, $A_F[u]=\n^2 u$ is the Hessian of $u$, $S_{k,F}[u]=S_k(\n^2 u)$ is the $k$-Hessian operator,  $I_{k,F}[u, \O]$ is the $k$-Hessian integral defined in \eqref{hess-integ}, and the above result reduces to Talenti-Tso's. On the other hand, when $k=1$, $S_{1,F}$ reduces to the so-called Finsler Laplacian and the above result reduces to Alvino et al.'s. The special case $n=k=2$ has been proved by the first two authors \cite{dpgmaan}. Our result completes this anisotropic type symmetrization.

For the Euclidean norm, the study of Talenti-Tso's symmetrization and the corresponding P\'olya-Szeg\H o type principle is based on Reilly's systematic work \cite{Reilly} on the elementary symmetric functions $S_k$ on symmetric matrices. 
Compared to the case when $F$ is the Euclidean norm, the new difficulty arises because of the anisotropy, which makes the anisotropic Hessian $A_F[u]$ a non-symmetric matrix. 
In the case $k=1$, the Finsler Laplacian is a quasilinear operator and in the case $n=k=2$, the 2-dimensional anisotropic Monge-Amp\`ere operator is possible to be computed directly. Here we complete such kind of symmetrization for any $k$ in any dimensions.
To achieve this, we have to make a systematic study of the elementary symmetric functions $S_k$ on non-symmetric matrices, which would be of independent interest. {In particular, we present a detailed proof of an important identity \eqref{difference} in Proposition \ref{prop-skij} for non-symmetric matrices by combinatoric technique.} See Section 2.1 below.  This is one of the main ingredients in our proof. 

The second main ingredient is to establish a relationship between the anisotropic curvature of the level sets and the anisotropic Hessian operators. For the Euclidean norm, such relation has been established by Tso \cite{Tso} and plays an important role in dealing with the Hessian operator. For the general norm, Wang and the third author \cite{WX} have established the relation between the anisotropic mean curvature and the Finsler Laplacian. In this paper we complete all other cases, see Section 2.4 below.

The rest of this paper is organized as follows.  Section 2 is devoted to make preparation for the symmetrizations. We first study the elementary symmetric functions on non-symmetric matrices  and the anisotropic Hessian operator intensively.  We then establish a relation between the anisotropic Hessian operator and the anisotropic curvature of the level sets, which might be of independent interest (see Theorem \ref{thm-curv}). The anisotropic Hessian operator on anisotropic radial functions is also studied (see Proposition \ref{skrad}).  In Section 3, we first review basic properties of mixed volumes and then define the anisotropic symmetrization. The P\'olya-Szeg\H o type principle is proved. Finally,  in Section 4, we make further applications of our symmetrizations to a comparison result for anisotropic Hessian equations and sharp Sobolev-type inequalities.




\

{
\noindent{\bf Acknowledgements.} We would like to thank the anonymous referees for their careful reading and critical comments which help to improve the exposition of this paper.
}

\

\section{Anisotropic Hessian operator}
\subsection{Preliminaries on anisotropy}\

Let $F\in C^{3}(\R^{n}\setminus\{0\})$ be a {\it strongly convex norm} on $\R^n$, in the sense that
\begin{itemize}
\item[(i)] $F$ is a norm in $\mathbb{R}^{n}$, i.e.,  $F$ is a convex, $1$-homogeneous function satisfying $F(x)>0$ when $x\neq 0$ and $F(0)=0$;
\item[(ii)] $F$ satisfies a uniformly elliptic condition: $\n^2 (\frac12 F^2)$ is positive definite in $\mathbb{R}^{n}\setminus \{0\}$.
\end{itemize} 
The {\it polar function} $F^o\colon \Rn \rightarrow [0,+\infty[$ 
of $F$ is defined as
\begin{equation*}
F^o(x)=\sup_{\xi\ne 0} \frac{\langle \xi, x\rangle}{F(\xi)}.
\end{equation*}
Then { $F^o\in C^{2}(\R^{n}\setminus\{0\})$} and it is also a strongly convex norm on $\R^{n}$ (see \cite{schn}).
Furthermore, 
\begin{equation*}
F(\xi)=\sup_{x\ne 0} \frac{\langle \xi, x\rangle}{F^o(x)}.
\end{equation*}
We remark that, throughout this paper,  we use conventionally $\xi$ as the variable for $F$ and $x$ as the variable for $F^o$ and $u$. 

Denote 
\[
\mathcal W= \{  x\in  \Rn \colon F^o(x)< 1 \}. 
\]
We call $\mathcal W$ the {\it unit Wulff ball} centered at the origin, and $\p \mathcal W$ the {\it Wulff shape}. 
More generally, we denote $$\mathcal W_r(x_0)=r\mathcal W+x_0,$$ and call it the {\it Wulff ball of radius $r$ centered at $x_0$}. We simply denote $\mathcal W_r=\mathcal W_r(0)$. 


The following properties of $F$ and $F^o$ hold true (see e.g. \cite{CS}):
for any  $x, \xi \in \Rn\setminus \{0\}$,
\begin{gather*}
\label{prima}
 \langle \n F(\xi) , \xi \rangle= F(\xi), \quad  \langle \n F^{o} (x), x \rangle
= F^{o}(x)
 \\
 \label{seconda} F(\n F^o(x))=F^o(\n F(\xi))=1,\\
\label{terza} 
F^o(x)  \n F(\n F^o(x) ) =x, \quad F(\xi) 
 \n F^o\left(\n F(\xi) \right) = \xi.\\
 {\n^2 F(\xi)\xi=0, \quad \n^3 (\frac12F^2)(\xi)\xi=0.}
 \end{gather*}

\smallskip

\subsection{Invariants on non-symmetric matrices}\

 Let $1\le k\le n$, be an integer. For a $n$-vector $\l=(\l_1,\cdots,\l_n)\in \rr^n$, the $k$-th elementary symmetric function on $\l$ is defined  by 
\[\s_k(\l)=\sum_{1\le i_1<\cdots< i_k\le n} \l_{i_1}\cdots \l_{i_k}.\]

For an $n\times n$ matrix $A=(A_{ij})\in \rr^{n\times n}$, we define $S_k(A)$
 to be the sum of all its principal $k\times k$ minors, namely,
\begin{eqnarray}\label{sk}
&&S_k(A)=\frac{1}{k!}\sum_{1\le i_1, \cdots, i_{k}, j_1,\cdots,j_k\le n}\delta_{i_1\cdots i_k}^{j_1\cdots j_k}A_{i_1j_1}\cdots A_{i_kj_k},
\end{eqnarray}
 where $\delta_{i_1\cdots i_k}^{j_1\cdots j_k}$ is the generalized Kronecker symbol which is defined to be $+1$ (resp. $-1$) if  $i_1,\cdots, i_k$  are distinct and $(j_1,\cdots, j_k)$ is an even (resp. odd) permutation of $(i_1,\cdots, i_k)$ and to be $0$ in any other cases. A basic property for $\delta_{i_1\cdots i_k}^{j_1\cdots j_k}$ is that it is anti-symmetric about any two indices from $\{i_1, \cdots, i_k\}$ or $\{j_1, \cdots, j_k\}$.  We will use the convention that $S_0(A)=1$.
  We remark that we do not assume that $A$ is symmetric. In the case that $A$ is symmetric, the invariant $S_k(A)$ has been intensively investigated by Reilly \cite{Reilly}.

When the eigenvalues $\l_A$ of $A \in \rr^{n\times n}$ are all real, it is clear that $$S_k(A)=\s_k(\l_A).$$ This follows easily from the fact that $\l_A$ are the real roots of the characterization polynomial $$\det(\l I- A)=\sum_{k=0}^n (-1)^kS_k(A)\l^{n-k}=0.$$ 
This is always the case we are interested in. In fact, in this paper, $A$ will be the product of two symmetric matrices, whose eigenvalues are clearly all real. Hence, we will denote $S_k(\l)$ instead of $\s_k(\l)$ as the elementary symmetric function in the following.  We remark that for such family of matrices $A$, some basic properties for $S_{2}(A)$ have been investigated by Cianchi-Salani \cite[Section 3.2]{CS}.

We denote the $(k-1)$-th {\it Newton transformation} by
$$S_k^{ij}(A)=\frac{\p S_k(A)}{\p A_{ij}}.$$ From the definition \eqref{sk}, it is easy to see that
\begin{eqnarray}\label{skij}
&&S_k^{ij}(A)=\frac{1}{(k-1)!}\sum_{{1\le i_1, \cdots, i_{k-1}, j_1,\cdots,j_{k-1}\le n}}\delta_{i_1\cdots i_{k-1} i}^{j_1\cdots j_{k-1} j}A_{i_1j_1}\cdots A_{i_{k-1}j_{k-1}}.\end{eqnarray}
{ In particular, $S_1^{ij}=\delta_{ij}$.}

Let $A_i$, $i=1,\cdots,k$ be $k$ $n\times n$ matrices. The {\it mixed discriminant} of $\{A_i\}$ is defined by
\begin{eqnarray}\label{mix}
S_k(A_1, A_2, \cdots, A_k)=\frac{1}{k!}\sum_{1\le i_1, \cdots, i_{k}, j_1,\cdots,j_k\le n}\delta_{i_1\cdots i_k}^{j_1\cdots j_k}(A_1)_{i_1j_1}\cdots (A_k)_{i_kj_k}.
\end{eqnarray}

We have the following easy properties.
\begin{prop}
It holds that:
\begin{itemize}
\item[(i)] the mixed discriminant $S_k(A_1, A_2, \cdots, A_k)$ is  multilinear and totally symmetric about $A_1$, $A_2, \cdots, A_k$;
\item[(ii)] \begin{eqnarray}\label{mix1}
S_k(A)= S_k(A, A, \cdots, A);
\end{eqnarray}
\item[(iii)] 
\begin{eqnarray}\label{mix2}
S_k(\underbrace{A, \cdots, A}_{k-1}, B)=\frac1k \sum_{1\le i,j\le n}S_k^{ij}(A)B_{ij};
\end{eqnarray}
\item[(iv)] 
\begin{eqnarray}\label{mix3}
S_k(A+B)=\sum_{r=0}^k \binom{k}{r}S_k({\underbrace{A, \cdots, A}_{k-r}, \underbrace{B, \cdots, B}_{r}}).
\end{eqnarray}
\end{itemize}
\end{prop}
\begin{proof} They follow from the definitions \eqref{sk}, \eqref{skij} and \eqref{mix}.
\end{proof}

Next, we prove a crucial identity for the Newton transformation $S_k^{ij}(A)$.
\begin{prop} 
\label{prop-skij}
Let $k=0, 1,\cdots, n-1$. For an $n\times n$ matrix $A=(A_{ij})$, we have
\begin{eqnarray}
\label{difference}
&&S_{k+1}^{ij}(A)= S_{k}(A)\d_{ij}- \sum_l S_{k}^{il}(A)A_{jl}, \quad \forall\quad i, j=1,\cdots, n.
\end{eqnarray}
\end{prop}
\begin{proof} This is well-known and easy for symmetric matrices, see for example \cite[Proposition 1.2 (d)]{Reilly}. 
Now we prove it for general (non-symmetric) matrices.

\noindent{\textbf{Case I: $i=j$.}} 

In the following computation, we fix $i$. First, from \eqref{skij}, we have
\begin{eqnarray}\label{eq1}
S_{k+1}^{ii}(A)&=& \frac{1}{k!}\sum_{\substack{i_1, \cdots i_{k}\neq i
\\ j_1, \cdots j_{k}\neq i}}\delta_{i_1\cdots i_{k}}^{j_1\cdots j_{k}}A_{i_1j_1}\cdots A_{i_{k}j_{k}}\nonumber
\\&=&\frac{1}{k!}\sum_{\substack{i_1, \cdots i_{k}
\\ j_1, \cdots j_{k}}}\delta_{i_1\cdots i_{k}}^{j_1\cdots j_{k}}A_{i_1j_1}\cdots A_{i_{k}j_{k}} -\frac{1}{k!} k\sum_{\substack{i_2, \cdots i_{k}\neq i\\ j_2, \cdots j_{k}\neq i}} \delta_{i_2\cdots i_{k}}^{j_2\cdots j_{k}}A_{ii}A_{i_2j_2}\cdots A_{i_{k}j_{k}}\nonumber
\\&&-\frac{1}{k!} k(k-1)(-1)\sum_{\substack{i_2, i_3,\cdots i_{k}\neq i\\ j_1, j_3,\cdots j_{k}\neq i}}\delta_{i_2i_3\cdots i_{k}}^{j_1j_3\cdots j_{k}}A_{ij_1}A_{i_2 i}A_{i_3j_3}\cdots A_{i_{k}j_{k}}\nonumber
\\&=&S_k(A)+ I_1+ I_2.
\end{eqnarray}
{ In \eqref{eq1}, we have separated the summation $\sum\limits_{\substack{i_1, \cdots i_{k}
\\ j_1, \cdots j_{k}}}$ to be three categories,
\begin{eqnarray*}
&&\mathcal{I}_0=\{i_1, \cdots i_{k}\neq i, j_1, \cdots j_{k}\neq i\},\\ 
&&\mathcal{I}_1=\{i_l=j_l=i \hbox{ for some }l\},\\ 
&&\mathcal{I}_2= \{i_l=j_m=i \hbox{ for some }l\ne m\}.
\end{eqnarray*}
The summation for $\mathcal{I}_0$ yields $S_{k+1}^{ii}(A)$, while, the summation for $\mathcal{I}_1$ and $\mathcal{I}_2$ yields $I_1$ and  $I_2$ respectively.}

Second, again from \eqref{skij}, we have
\begin{eqnarray}\label{eq2}
\sum_l S_{k}^{il}(A)A_{il}&=& \frac{1}{(k-1)!}\sum_{\substack{l\\i_2, \cdots i_{k}
\\ j_2, \cdots j_{k}}}\delta_{i i_2\cdots i_{k}}^{l j_2\cdots j_{k}}A_{i_2j_2}\cdots A_{i_{k}j_{k}} A_{il}\nonumber
\\&=&\frac{1}{(k-1)!}\sum_{\substack{i_2, \cdots i_{k}\neq i
\\ j_2, \cdots j_{k}\neq i}}\delta_{i_2 \cdots i_{k}}^{j_2\cdots j_{k}}A_{i_2 j_2}A_{i_3j_3}\cdots A_{i_{k}j_{k}} A_{ii}\nonumber
\\&&+\frac{1}{(k-1)!}\sum_{\substack{i_2, i_3,\cdots i_{k}\neq i
\\ l,\ j_3, \cdots j_{k}\neq i}}(k-1)(-1)\delta_{i_2 i_3\cdots i_{k}}^{l\ j_3\cdots j_{k}}A_{i_2 i}A_{i_3j_3}\cdots A_{i_{k}j_{k}} A_{il}\nonumber
\\&=&I\!I_1+I\!I_2.
\end{eqnarray}
{ Note that for the summation $\sum\limits_{\substack{l\\i_2, \cdots i_{k}
\\ j_2, \cdots j_{k}}}\delta_{i i_2\cdots i_{k}}^{l j_2\cdots j_{k}}$, there is exactly one of $\{l, j_2\cdots j_{k}\}$ equal to $i$ (otherwise, the Kronecker symbol would be $0$).
Hence in \eqref{eq2}, we have separated the indices set  to be $\{l=i\}$ and $\{l\ne i\}$. In the case  $\{l=i\}$, other indices $i_2, \cdots i_{k},
j_2, \cdots j_{k}$ cannot be $i$, for which the summation yields $I\!I_1$, while in the case $\{l\ne i\}$, exactly one of $j_2\cdots j_{k}$ must be $i$, for which the summation yields $I\!I_2$.}

It is direct to see that $I_1=-I\!I_1$ and $I_2=-I\!I_2$. It follows from \eqref{eq1} and \eqref{eq2} that
\begin{eqnarray}\label{eqii}
&&S_{k+1}^{ii}(A)= S_k(A)- \sum_l S_{k}^{il}(A)A_{il}.
\end{eqnarray}
\noindent{\textbf{Case II: $i\neq j$.}} 

In the following computation, we fix $i$ and $j$.
First, from \eqref{skij}, we have
\begin{eqnarray}\label{eq3}
S_{k+1}^{ij}(A)&=& \frac{1}{k!}\sum_{\substack{i_1, \cdots i_{k}
\\ j_1, \cdots j_{k}}}\delta_{i i_1\cdots i_{k}}^{j j_1\cdots j_{k}}A_{i_1j_1}\cdots A_{i_{k}j_{k}}\nonumber
\\&=&\frac{1}{k!}k(-1)\sum_{\substack{i_2, \cdots i_{k}\neq i, j
\\ j_2, \cdots j_{k}\neq i, j}}\delta_{i_2\cdots i_{k}}^{j_2\cdots j_{k}}A_{ji}A_{i_2j_2}\cdots A_{i_{k}j_{k}} \nonumber\\&&+\frac{1}{k!} k(k-1)\sum_{\substack{i_2, \cdots i_{k}\neq i, j\\ j_2, \cdots j_{k}\neq i, j}} \delta_{i_2i_3\cdots i_{k}}^{j_1j_3\cdots j_{k}}A_{j j_1}A_{i_2 i}A_{i_3j_3}\cdots A_{i_{k}j_{k}}\nonumber
\\&=& I\!I\!I_1+ I\!I\!I_2.
\end{eqnarray}
{ Note that for the summation $\sum\limits_{\substack{i_1, \cdots i_{k}
\\ j_1, \cdots j_{k}}}\delta_{i i_1\cdots i_{k}}^{j j_1\cdots j_{k}}$, there are exactly one of $\{j_1, j_2\cdots j_{k}\}$ equal to $i$ and  exactly one of $\{i_1, i_2\cdots i_{k}\}$ equal to $j$ (Otherwise, the Kronecker symbol would be $0$).
. Hence in \eqref{eq3} we have separated the index set to be two categories:
\begin{eqnarray*}
&&\mathcal{III}_1=\{j_l=i, i_l=j, \hbox{ for some }l\},\\
&&\mathcal{III}_2=\{j_l=i, i_m=j, \hbox{ for some }l\neq m\}.
\end{eqnarray*}
The summation for $\mathcal{III}_1$ and $\mathcal{III}_2$ yields $ I\!I\!I_1$ and  $I\!I\!I_2$ respectively.}

Second, again from \eqref{skij}, we have
{ \begin{eqnarray}\label{eq4}
\sum_l S_{k}^{il}(A)A_{jl}&=& \frac{1}{(k-1)!}\sum_{\substack{l\\i_2, \cdots i_{k}
\\ j_2, \cdots j_{k}}}\delta_{i i_2\cdots i_{k}}^{l j_2\cdots j_{k}}A_{i_2j_2}\cdots A_{i_{k}j_{k}} A_{jl}\nonumber
\\&=&\frac{1}{(k-1)!}\sum_{\substack{i_2, \cdots i_{k}\neq i
\\ j_2, \cdots j_{k}\neq i}}\delta_{i_2 \cdots i_{k}}^{j_2\cdots j_{k}}A_{i_2 j_2}A_{i_3j_3}\cdots A_{i_{k}j_{k}} A_{ji}\nonumber
\\&&+\frac{1}{(k-1)!}\sum_{\substack{i_2, i_3,\cdots i_{k}\neq i
\\ l,\ j_3, \cdots j_{k}\neq i}}(k-1)(-1)\delta_{i_2 i_3\cdots i_{k}}^{l\ j_3\cdots j_{k}}A_{i_2 i}A_{i_3j_3}\cdots A_{i_{k}j_{k}} A_{jl}\nonumber
\\&=&\frac{1}{(k-1)!}\sum_{\substack{i_2, \cdots i_{k}\neq i
\\ j_2, \cdots j_{k}\neq i}}\delta_{i_2 \cdots i_{k}}^{j_2\cdots j_{k}}A_{i_2 j_2}A_{i_3j_3}\cdots A_{i_{k}j_{k}} A_{ji}\nonumber
\\&&+\frac{1}{(k-1)!}\sum_{\substack{i_2, i_3,\cdots i_{k}\neq i
\\ j_2, j_3, \cdots j_{k}\neq i}}(k-1)(-1)\delta_{i_2 i_3\cdots i_{k}}^{j_2 j_3\cdots j_{k}}A_{i_2 i}A_{i_3j_3}\cdots A_{i_{k}j_{k}} A_{jj_2}.
\end{eqnarray}
The separation  for the indices set in \eqref{eq4} is the same as that in \eqref{eq2}. Next we separate the indices set 
$\{i_2, \cdots i_{k}, j_2, \cdots j_{k}\neq i\}$ to be \begin{eqnarray*}
&&\{i_2, \cdots i_{k}, j_2, \cdots j_{k}\neq i, j\},\\
&&\{i_2, \cdots i_{k}, j_2,\cdots, j_k\neq i,\hbox{ and one of these indices is } j\}
\end{eqnarray*}
so that we get from \eqref{eq4}
\begin{eqnarray}\label{eq5}
\sum_l S_{k}^{il}(A)A_{jl}&=&\frac{1}{(k-1)!}\sum_{\substack{i_2, \cdots i_{k}\neq i, j
\\ j_2, \cdots j_{k}\neq i, j}}\delta_{i_2 \cdots i_{k}}^{j_2\cdots j_{k}}A_{i_2 j_2}A_{i_3j_3}\cdots A_{i_{k}j_{k}} A_{ji}\nonumber
\\&&+\frac{1}{(k-1)!}\sum_{\substack{i_2, \cdots i_{k}\neq i, j
\\ j_2, \cdots j_{k}\neq i, j}}(k-1)(-1)\delta_{i_2\cdots i_{k}}^{j_2\cdots j_{k}}A_{i_2 i}A_{i_3j_3}\cdots A_{i_{k}j_{k}} A_{jj_2}\nonumber
\\&&+\frac{1}{(k-1)!}\hspace{-1cm}\sum_{\substack{i_2, \cdots i_{k}, j_2,\cdots, j_k\neq i,
\\[.1cm]  \hbox{\scriptsize{one of these indices is} } j}}\hspace{-1cm} \delta_{i_2\cdots i_{k}}^{j_2\cdots j_{k}}\left[A_{ji}A_{i_2 j_2}- (k-1)A_{i_2 i}A_{jj_2}\right]A_{i_3j_3}\cdots A_{i_{k}j_{k}}\nonumber
\\[.2cm]&=& I\!V_1+ I\!V_2 + I\!V_3.
\end{eqnarray}
}
It is direct to see $I\!I\!I_1 =-I\!V_1$ and $I\!I\!I_2= -I\!V_2$. We claim that
\begin{eqnarray}
I\!V_3= \sum_{\substack{i_2, \cdots i_{k}, j_2,\cdots, j_k\neq i,
\\ \hbox{\scriptsize{one of these indices is} } j}} \delta_{i_2\cdots i_{k}}^{j_2\cdots j_{k}}\left[A_{ji}A_{i_2 j_2}- (k-1)A_{i_2 i}A_{jj_2}\right]A_{i_3j_3}\cdots A_{i_{k}j_{k}} =0.
\end{eqnarray}
If the claim is true, we see from \eqref{eq3} and  \eqref{eq4} that 
\begin{eqnarray}\label{eqij}
&&S_{k+1}^{ij}(A)=- \sum_l S_{k}^{il}(A)A_{jl}, \quad \hbox{ if }i\neq j.
\end{eqnarray}
Next we prove the claim.
We seperate the summation in $I\!V_3$ into the following five cases: 

(i) $i_2=j_2=j$, and the summand gives
\begin{eqnarray*}
\sum_{\substack{i_3, \cdots i_{k}\neq i, j
\\ j_3, \cdots j_{k}\neq i, j}} \delta_{i_3\cdots i_{k}}^{j_3\cdots j_{k}}(-(k-2))A_{jj}A_{ji}A_{i_3j_3}\cdots A_{i_{k}j_{k}}: =I\!V_{31};
\end{eqnarray*}

(ii) $i_2=j$, only one of $j_r=j$ for $r=3,\cdots, k$, and the summand gives
\begin{eqnarray*}
(k-2)\hspace{-.6cm}\sum_{\substack{i_3, i_4, \cdots i_{k}\neq i, j
\\ j_2, j_4, \cdots j_{k}\neq i, j}}(-1) \delta_{i_3 i_4\cdots i_{k}}^{j_2 j_4\cdots j_{k}}(-(k-2))A_{jj_2}A_{ji}A_{i_3j}A_{i_4 j_4}\cdots A_{i_{k}j_{k}}: =I\!V_{32};
\end{eqnarray*}

(iii) $j_2=j$, only one of $i_r=j$ for $r=3,\cdots, k$, and the summand gives
\begin{eqnarray*}
(k-2)\hspace{-.6cm}\sum_{\substack{i_2, i_4, \cdots i_{k}\neq i, j
\\ j_3, j_4, \cdots j_{k}\neq i, j}}(-1) \delta_{i_2 i_4\cdots i_{k}}^{j_3 j_4\cdots j_{k}}\left[A_{ji}A_{i_2 j}- (k-1)A_{i_2 i}A_{jj}\right]A_{jj_3}A_{i_4 j_4}\cdots A_{i_{k}j_{k}}: =I\!V_{33}+I\!V_{34};
\end{eqnarray*}

(iv) $i_r=j_r=j$ for $r=3,\cdots, k$, and the summand gives
\begin{eqnarray*}
(k-2)\hspace{-.6cm}\sum_{\substack{i_2, i_4 \cdots i_{k}\neq i, j
\\ j_2, j_4\cdots j_{k}\neq i, j}} \delta_{i_2 i_4\cdots i_{k}}^{j_2 j_4\cdots j_{k}}\left[A_{ji}A_{i_2 j_2}- (k-1)A_{i_2 i}A_{jj_2}\right]A_{jj}A_{i_4 j_4}\cdots A_{i_{k}j_{k}}: =I\!V_{35}+I\!V_{36};
\end{eqnarray*}

(v) $i_r=j$ and $j_s=j$ for $r, s=3,\cdots, k$ with $r\neq s$, and the summand gives
\begin{eqnarray*}
(k-2)(k-3)\hspace{-.7cm}\sum_{\substack{i_2, i_4, i_5, \cdots i_{k}\neq i, j
\\ j_2, j_3, j_5\cdots j_{k}\neq i, j}}(-1) \delta_{i_2 i_4 i_5\cdots i_{k}}^{j_2 j_3 j_5\cdots j_{k}}\left[A_{ji}A_{i_2 j_2}- (k-1)A_{i_2 i}A_{jj_2}\right]A_{jj_3}A_{i_4j}A_{i_5j_5}\cdots A_{i_{k}j_{k}}\\: =I\!V_{37}+I\!V_{38};
\end{eqnarray*}

That is, \[I\!V_3= \sum_{\a=1}^8 I\!V_{3\a}.\]
It is direct to check that $I\!V_{31}+I\!V_{35}=0$, $I\!V_{32}+I\!V_{33}+ I\!V_{37}=0$ and $I\!V_{34}+  I\!V_{36}=0$. Finally, $I\!V_{38}=0$ since the Kronecker symbol $\delta_{i_2 i_4 i_5\cdots i_{k}}^{j_2 j_3 j_5\cdots j_{k}}$ is anti-symmetric with respect to $j_2$ and $j_3$ while $A_{i_2 i}A_{jj_2}A_{jj_3}A_{i_4j}A_{i_5j_5}\cdots A_{i_{k}j_{k}}$ is symmetric with respect to $j_2$ and $j_3$.
We get the claim that $I\!V_3 =0$ and in turn \eqref{eqij}.

Our assertion follows from \eqref{eqii} and \eqref{eqij}.
\end{proof}

\smallskip

\subsection{Anisotropic Hessian operator}\

Let $F\in C^{3}(\R^{n}\setminus\{0\})$ be a strongly convex norm on $\R^n$. Let $\O\subset \rr^n$ be an open bounded domain and $u\in C^3(\O)$.
We denote by $F_i, F_{ij}, \ldots$  the partial derivatives of $F$ and by $u_i, u_{ij},\ldots$ the partial derivatives of $u$, 
$$F_i=\frac{\p F}{\p \xi_i},  F_{ij}=\frac{\p^2 F}{\p \xi_i\p \xi_j}, \quad u_i=\frac{\p u}{\p x_i}, u_{ij}=\frac{\p^2 u}{\p x_i\p x_j}.$$

Denote by $A_F[u]=((A_F)_{ij}[u])$ the matrix \begin{eqnarray}\label{A_F}
(A_F)_{ij}[u]&:=& \p_{x_j}\left[\p_{\xi_i}\left(\frac12F^2\right)(\n u)\right]=\sum_l\left(\frac12F^2\right)_{il}(\n u) u_{lj}\nonumber\\&=& \sum_lF_i(\n u)F_l(\n u)u_{lj}+F(\n u)F_{il}(\n u)u_{lj}, \hbox{ when }\n u\neq 0.
\end{eqnarray}
We regard $A_F[u]=0$ when $\n u=0$, in the case that $F$ is not the Euclidean norm.

The anisotropic $k$-Hessian operator of $u$ is defined as 
\begin{equation*}
S_{k,F}[u]:=S_{k}(A_F[u]).
\end{equation*}
The corresponding Newton transformation will be written as \[S_{k, F}^{ij}[u]= S_k^{ij}(A_F[u]).\]
Note that in general $S_{k, F}^{ij}[u]$ is not symmetric about $i$ and $j$.

When $F$ is the Euclidean norm, $A_F[u]$ reduces to the Hessian $\n^2 u$ and $S_{k,F}[u]$  reduces to the classical Hessian operator $S_k[u]=S_k(\n^2 u)$. When $k=1$, $S_{k,F}[u]$ reduces to the Finsler-Laplacian $$\De_F u=\div\left(\n \left(\frac12F^2\right)(\n u)\right),$$ which has been widely investigated in recent decades, see e.g. \cite{CS, WX}. For $k=2$, some properties of $S_{2, F}[u]$ have been investigated by \cite{CS}.

For notation simplicity, we omit the subscription $F$ in $A_F[u]$ and $S_{k,F}[u]$, and $(\n u)$ in $F(\n u),$ $F_i(\n u), \ldots$, when no confusion occurs.
We will make an overall study on $S_{k,F}[u]$ based on the properties of $S_k(A)$ for non-symmetric matrices.
We show several point-wise identities about $S_{k,F}[u]$. We remark that these point-wise identities only hold when $\n u\ne 0$.

The first property we shall prove is the following divergence free property.
\begin{prop}\label{divfree} 
It holds that
\begin{eqnarray}
\label{divfreeeq}
\sum_j\p_j S_{k}^{ij}[u]=0, \quad \forall i=1,\ldots, n.
\end{eqnarray}
\end{prop}
\begin{proof}
From $\eqref{skij}$, we get
\begin{eqnarray}
\sum_j\p_j S_{k}^{ij}[u]\!\!&=&\!\!\frac{1}{(k-1)!}\,(k-1)\sum_{\substack{j\\i_1\cdots j_{k-1}}}\delta_{i i_1\cdots i_{k-1}}^{j j_1\cdots j_{k-1}}\left( \p_j A_{i_1j_1}[u]\right)A_{i_2j_2}[u]\cdots A_{i_{k-1} j_{k-1}}[u]\nonumber
\\&=&\!\!\frac{1}{(k-2)!}\!\sum_{\substack{j\\i_1\cdots j_{k-1}}} \!\!\delta_{i i_1\cdots i_{k-1}}^{j j_1\cdots j_{k-1}} \left[(\frac12F^2)_{i_1 l m}u_{mj} u_{lj_1}+(\frac12F^2)_{i_1 l}u_{lj_1 j}\right]A_{i_2j_2}[u]\cdots A_{i_{k-1} j_{k-1}}[u]\nonumber
\end{eqnarray}
It is easy to see that $(\frac12F^2)_{i_1 l m}u_{mj} u_{lj_1}+(\frac12F^2)_{i_1 l}u_{lj_1 j}$ is symmetric with respect to $j$ and $j_1$.
Since $\delta_{i i_1\cdots i_{k-1}}^{j j_1\cdots j_{k-1}}$ is anti-symmetric with respect to $j$ and $j_1$, we get the assertion.\end{proof}

\begin{prop}\label{skb} Let $A[u]=B[u]+C[u],$ where 
\[
B_{ij}[u]=\sum_lFF_{il}u_{lj}, \quad C_{ij}[u]=\sum_lF_i F_l u_{lj}.
\]
 Then
\begin{eqnarray*}
S_k(B[u])=\frac{1}{F}\sum_{i,j}S_{k+1}^{ij}[u] u_jF_i.
\end{eqnarray*}
\end{prop}
\begin{proof}
Using Proposition \ref{prop-skij}, we have
\begin{eqnarray}\label{xeq1'}
\sum_{i,j}\frac{S_{k+1}^{ij}[u]u_jF_i}{F} &=& \sum_{i,j}\frac1F S_k[u]\d_{ij}u_jF_i-\sum_{i,j,l}\frac1F S_k^{il}[u]A_{jl}[u]u_j F_i \nonumber
\\&=&S_k[u]-\sum_{i,m,l}S_k^{il}[u]F_mu_{ml}F_i\nonumber
\\&=&S_k[u]-\sum_{i,l}S_k^{il}[u]C_{il}[u].
\end{eqnarray}
In the second equality we used $\sum_iF_i u_i =F$ and $\sum_jF_{ij} u_j=0$.

On the other hand, since $B[u]=A[u]-C[u]$, we compute
\begin{align}\label{xeq2}
\hspace{-1.2cm} S_k (B[u])&= S_k(A[u]-C[u])\nonumber
\\&=\sum_{r=0}^k \binom{k}{r} (-1)^r S_k(\underbrace{A[u], \cdots, A[u]}_{k-r}, \underbrace{C[u], \cdots, C[u]}_{r})\nonumber
\\&=S_k[u]- \sum_{i,j}S_k^{ij}[u]C_{ij}[u] + \sum_{r=2}^k \binom{k}{r} (-1)^r S_k(\underbrace{A[u], \cdots, A[u]}_{k-r}, \underbrace{C[u], \cdots, C[u]}_{r}).
\end{align}
We claim 
\begin{eqnarray*}
S_k(\underbrace{A[u], \cdots, A[u]}_{k-r}, \underbrace{C[u], \cdots, C[u]}_{r})=0, \hbox{ for }2\le r\le k.
\end{eqnarray*}
If the claim is true, then the assertion follows from \eqref{xeq1'} and \eqref{xeq2}.

We prove the claim. Using \eqref{mix3}, \eqref{mix2} and the definition of $C[u]$,
\begin{eqnarray}\label{xeq3}
&& S_k (\underbrace{A[u], \cdots, A[u]}_{k-r}, \underbrace{C[u], \cdots, C[u]}_{r})\nonumber
\\&=&\frac{1}{k!}\sum_{\substack{i_1, \cdots i_{k}
\\ j_1, \cdots j_{k}}}\delta_{i_1\cdots i_{k}}^{j_1\cdots j_{k}}A_{i_1j_1}\cdots A_{i_{k-r} j_{k-r}} C_{i_{k-r+1} j_{k-r+1}}\cdots C_{i_k j_k}\nonumber
\\&=&\frac{1}{k!}\sum_{\substack{l, m\\i_1, \cdots i_{k}
\\ j_1, \cdots j_{k}}}\delta_{i_1\cdots i_{k}}^{j_1\cdots j_{k}}A_{i_1j_1}\cdots A_{i_{k-r} j_{k-r}}C_{i_{k-r+1} j_{k-r+1}}\cdots C_{i_{k-2} j_{k-2}}F_{i_{k-1}}F_l u_{l j_{k-1}}F_{i_{k}}F_m u_{m j_{k}}.
\end{eqnarray}
Since $\delta_{i_1\cdots i_{k}}^{j_1\cdots j_{k}}$ is anti-symmetric with respect to $i_{k-1}$ and $i_k$, while 
\[
A_{i_1j_1}\cdots A_{i_{k-r} j_{k-r}} C_{i_{k-r+1} j_{k-r+1}}\cdots C_{i_{k-2} j_{k-2}} F_{i_{k-1}}F_l u_{l j_{k-1}}F_{i_{k}}F_m u_{m j_{k}}
\] is symmetric with respect to $i_{k-1}$ and $i_k$, we conclude that the summation on the right hand side of \eqref{xeq3} is zero.
The proof is completed.
\end{proof}

\smallskip

\subsection{Anisotropic $k$-th mean curvature of level sets}\

Let $M$ be a smooth closed hypersurface in $\R^{n}$ and $\nu$ be the unit Euclidean outer normal of $M$. The anisotropic outer normal of $M$ is defined by
\[
\nu_{F}=\n F(\nu).
\]
The anisotropic principal curvatures $\k_F=(\kappa_{1}^{F},\ldots,\kappa_{n-1}^{F})\in\R^{n-1}$ are  defined as  the eigenvalues of the map
\[
d\nu_{F}\colon T_{p}M \to T_{\nu_{F}(p)}\mathcal W.
\]
For $k=1,\ldots,n$ the anisotropic $k$-th mean curvature of $M$ is $S_{k}(\kappa_{F})$. See e.g. \cite{WX}.

Let $u\in C(\bar \O)\cap C^2(\O)$ and $m=\min u, M=\max u$. Denote
$$\O_t=\{x: u(x)<t\}, \qquad \Sigma_t=\{x: u(x)=t\}, \hbox{ for }t\in[m, M].$$
We call $\S_t$ is {\it non-degenerate} if $\n u\ne 0$ on $\S_t$. 
We shall establish the relation between the anisotropic $k$-th mean curvatures of $\S_t$ and the anisotropic $k$-th Hessian operator.
\begin{thm}\label{thm-curv} Assume $\S_t$ is a non-degenerate level set of $u$. Then the anisotropic $k$-th mean curvature $S_k(\k_F)$ of $\Sigma_t$ 
  satisfies 
\begin{eqnarray}\label{k-curv}
S_k(\k_F)=S_k \left(\sum_lF_{il}u_{lj}\right)=\frac{1}{{F^{k+1}}}\sum_{i,j}{S_{k+1}^{ij}[u]u_j F_i},
\end{eqnarray}
\end{thm}

\begin{rem}In the case $F$ is the Euclidean norm, \eqref{k-curv} reduces to 
\begin{eqnarray*}
\s_k(\k)=\sum_{i,j}\frac{S_{k+1}^{ij}(\n^2 u)u_iu_j}{|\n u|^{k+2}},
\end{eqnarray*}
which has been proved by Tso \cite[p.100, Equation (6)]{Tso}.
In the case $k=1$, \eqref{k-curv} reduces to 
\[
H_F=\sum_{i,j}F_{ij}u_{ij}=\frac{1}{F}\Bigg(\De_F u-\sum_{i,j}F_iF_j u_{ij}\Bigg)
\]
which has been proved by Wang-Xia \cite[Theorem 3]{WX}.
\end{rem}
\begin{proof}
Fix $p\in \S_t$. The unit normal and the anisotropic normal at $p$ are given by$$\nu=\frac{\n u}{|\n u|}, \quad \nu_F=\n F(\nu)=\n F(\n u)$$ respectively.

Let $\{\varepsilon_i\}_{i=1}^n$ be the canonical orthonormal basis of $\RR^n$. Write \[\nu_F=\sum_i\nu_F^i \varepsilon_i=\sum_i F_i  \varepsilon_i.\]
 $d\nu_F|_p: T_p\RR^n\to T_{\nu(p)} \RR^n$ is a linear transformation given by:
\begin{eqnarray}\label{eigeneq1}
d\nu_F|_p(\ve_j)=\sum_i\p_j(F_i(\n u))|_p \ve_i= \sum_{i,l} F_{il}u_{lj}|_p \ve_i.
\end{eqnarray}
Thus the eigenvalues of $d\nu_F|_p: T_p\RR^n\to T_{\nu_F(p)}\RR^n$ are given by the eigenvalues of matrix $(\sum_l F_{il}u_{lj}|_p)$.

Next, we compute the eigenvalue of $d\nu_F|_p: T_p\RR^n\to T_{\nu(p)}\RR^n$ in another way. 
By rotation of coordinates, we can assume that at $p$, the $x_n$ coordinate axis lies in the direction of $\nu(p)$. In turn, the $x_1$ to $x_{n-1}$ coordinates axes span the tangent space $T_p \Sigma_t$.
In some neighborhood $\mathcal{N}_p$ of $p$, $\Sigma_t$ can be represented by a graph $$x_n=\varphi(x'), \quad x'=(x_1,\cdots, x_{n-1}),$$ where $\varphi\in C^2(T_p \Sigma_t\cap \mathcal{N}_p)$ is such that $$\varphi(p')=0,\quad D\varphi(p')=0.$$
Here $D$ denotes the gradient on $\rr^{n-1}$.

In this local coordinate, the unit normal of $\Sigma_t\cap \mathcal{N}_p$ is given by
\[\nu=\frac{1}{\sqrt{1+|D\varphi|^2}}\left(-D\varphi, 1\right).\]
It is direct to verify that at $p$,
\begin{eqnarray*}
&\p_\a \nu^\b|_p= -\p_\a\p_\b\varphi, \quad &\a, \b=1,\cdots, n-1.\\
&\p_i \nu^j|_p=0, \quad & i=n \hbox{ or }j=n.
\end{eqnarray*}
Thus,
$d\nu|_p: T_p\RR^n\to T_{\nu(p)}\RR^n$ under the basis $\{\p_i\}_{i=1}^n$ is given by the matrix
\begin{equation*}
\left(
\begin{array}{rccl}
-D^2 \varphi|_p & & 0\\
0 & &0
\end{array}
\right).
\end{equation*}
Hence, 
$d\nu_F|_p: T_p\RR^n\to T_{\nu_F(p)}\RR^n$ under the basis $\{\p_i\}_{i=1}^n$ is given by the matrix
\begin{equation*}
 \n^2 F(\nu|_p)\cdot
\left(
\begin{array}{rccl}
-D^2 \varphi|_p & & 0\\
0 & &0
\end{array}
\right).
\end{equation*}
Combining the above two ways of computation, we see
\begin{eqnarray}\label{curv-1}
S_k\left[(\sum_l F_{il}u_{lj}|_p)\right]= S_k\left[ \n^2 F(\nu|_p)\cdot
\left(
\begin{array}{rccl}
-D^2 \varphi|_p & & 0\\
0 & &0\end{array}
\right)
 \right].
\end{eqnarray}

On the other hand, Recall that the anisotropic principal curvatures $\k_F$ are defined as the eigenvalues of $d\nu_F: T_p \Sigma_t\to T_{\nu_F(p)} \mathcal{W}_F$.
Thus for a local frame $\{e_\a\}_{\a=1}^{n-1} \subset T\Sigma_t$,
\[d\nu_F(e_\a)= \sum_{\b,\g}g^{\b\g}\<d\nu_F(e_\a), e_\g\> e_\b= \sum_{\b,\g, \eta}\sum_{i,j}g^{\b\g}F_{ij}(\nu) h_{\a}^{\eta} e_\eta^j e_\g^i e_\b.\]
In particular, at $p$, since $e_\a=\p_\a$, $g_{\a\b}=\d_{\a\b}$ and $h_{\a}^\eta=-\p_\a\p_\eta \varphi$, $e_\a^i =\d_{i \a}$, we have
\[d\nu_F|_p(\p_\a)=\sum_{\b, \eta}F_{\b\eta}(\nu) (-\p_\a\p_\eta \varphi) \p_\b.\]
Thus the eigenvalues of the $(n-1)\times(n-1)$ matrix $(\sum_{\eta}F_{\b\eta}(\nu) (-\p_\a\p_\eta \varphi)|_p)$ are exactly the anisotropic principal curvatures $\k_F$ at $p$.
Therefore, we get
\begin{eqnarray}\label{curv-2}
S_k(\k_F)=S_k\left[\sum_{\eta}F_{\b\eta}(\nu) (-\p_\a\p_\eta \varphi)|_p)\right].
\end{eqnarray}

It follows from \eqref{curv-1} and \eqref{curv-2} that
\begin{eqnarray}\label{eigeneq2}
&&S_k\left[ \n^2 F(\nu)\cdot
\left(
\begin{array}{rccl}
-D^2 \varphi & & 0\\
0 & &0\end{array}
\right)
 \right]=S_k\left[ 
\left(
\begin{array}{rccl}
\left(\sum_{\eta}F_{\b\eta}(\nu)(-\p_\a\p_\eta \varphi)\right) & & 0\\
0 & &0\end{array}
\right)
 \right]=S_k(\k_F).
\end{eqnarray}
The first identity in \eqref{k-curv} is proved. The second identity in \eqref{k-curv} follows easily from Proposition \ref{skb}.
The proof is completed.
\end{proof}


\begin{cor} Assume $\S_t$ is a non-degenerate level set of $u$. Then
\begin{eqnarray}
\label{kcurvformula}
S_{k}[u]= S_{k}(\k_{F})F^{k} +\frac1F\sum_{i,j,l}S_{k}^{ij}[u]F_{i}u_{l}A_{lj}[u].
\end{eqnarray}
\end{cor}
\begin{proof}
Using \eqref{k-curv} and \eqref{difference}, we obtain that
\begin{eqnarray*}
&&S_{k}(\kappa_{F})F^{k} +\frac1F\sum_{i,j,l}S_{k}^{ij}[u]F_{i}u_{l}A_{lj}[u]
\\ &=&\frac{1}{F}\sum_{i,j}{S_{k+1}^{ij}[u]u_j F_i} + \frac1F\sum_{i,l}\left( S_{k}[u]\d_{il}-  S_{k+1}^{il}[u]\right)F_{i}u_{l} \\
&=& \frac1FS_k[u]\sum_i F_iu_i =S_k[u].
\end{eqnarray*}
\end{proof}

\begin{prop} Let $u\in C(\bar \O)\cap C^{{n}}(\O)$ and $u=0$ on $\p\O$. Then
\begin{multline}
\label{byparts}
\int_{\O} (-u)S_k[u] \, dx = \frac1k\int_{\O}  \sum_{i,j}S_k^{ij}[u]FF_iu_j \, dx=\frac1k\int_{m}^M \int_{\S_t}S_{k-1}(\k_F)F^{k}(\n u)F(\nu) \, d\mathcal{H}^{n-1} dt.
\end{multline}
\end{prop}
\begin{proof} 
Using \eqref{mix1}, \eqref{mix2} and \eqref{divfreeeq}, we have
\begin{eqnarray}\label{xeq0}
kS_k[u]&=&\sum_{i,j} S_k^{ij}[u]A_{ij}[u]=\sum_{i,j,l}S_k^{ij}[u] (\frac12F^2)_{il}u_{lj}\nonumber
\\&=&\sum_{i,j,l}\left[\p_j\left(S_k^{ij}[u] (\frac12F^2)_{il}u_{l}\right)-  S_k^{ij}[u]\sum_m(\frac12F^2)_{ilm}u_{mj}u_l\right]\nonumber
\\&=&\sum_{i,j}\p_j\left(S_k^{ij}[u] FF_i\right).
\end{eqnarray}
Last equality above holds because $\sum_i F_iu_i=F$, $\sum_lF_{il}u_l=0$ and $\sum_l(\frac12F^2)_{ilm}u_l=0$.

Multiplying \eqref{xeq0} by $(-u)$ and integrating over $\O$, noting that $u=0$ on $\p\O$, we get by integration by parts,
the first identity in \eqref{byparts}.

{Now, by Sard's theorem, $\S_t$ is non-degenerate and a $C^2$ hypersurface in $\rr^n$ for a.e. $t\in [m, M]$.} Hence, using the co-area formula, we have
\begin{eqnarray*}
\int_{\O}  \sum_{i,j}S_k^{ij}[u]FF_iu_j \, dx&=& \int_{m}^M\int_{\S_t}  \sum_{i,j}S_k^{ij}[u]FF_iu_j\frac{1}{|\n u|} \, d\mathcal{H}^{n-1} dt
\\&=&\int_{m}^M \int_{\S_t}S_{k-1}(\k_F)F^{k}(\n u)F(\nu) \, d\mathcal{H}^{n-1} dt.
\end{eqnarray*}
In the last equality we used \eqref{k-curv} and the fact that $\nu=\frac{\n u}{|\n u|}$ on $\S_t$. This is the second identity in \eqref{byparts}.
The proof is completed.
\end{proof}

\subsection{Anisotropic radial functions}\

In this subsection,  we compute  the $k$-Hessian anisotropic operator for anisotropic radial functions, namely, functions which are symmetric with respect to $F^o$.
\begin{prop}
\label{skrad}
Let $u(x)=v(r)$, where $r=F^o(x)$. Then 
\begin{align*}
S_k[u]&= \binom{n-1}{k-1}\frac{v''(r)}{r}\left(\frac{v'(r)}{r}\right)^{k-1}+\binom{n-1}{k}\left(\frac{v'(r)}{r}\right)^k \\& =\binom{n-1}{k-1}r^{-(n-1)} \left(\frac{r^{n-k}}{k}(v'(r))^{k}\right)',
\end{align*}and
\begin{align*}
\sum_{i,j}S_k^{ij}[u]FF_iu_j&= \binom{n-1}{k-1}r^{-(k-1)}v'(r)^{k+1}.
\end{align*}
\end{prop}

\begin{proof}
It is direct to compute that 
\[\p_{x_i} u(x)=v'(r)\p_{x_i} F^o(x).\]
\[\p_{\xi_i}\left(\frac12F^2\right)(\n u)(x)= FF_i(\n u)(x)=v'(r)\frac{x_i}{F^o(x)}=\frac{v'(r)}{r}x_i,\]
\[A_{ij}[u](x)=\p_{x_j}\left[\p_{\xi_i}\left(\frac12F^2\right)(\n u)\right](x)=\frac{v'(r)}{r}\delta_{ij}+\left(\frac{v''(r)}{r}-\frac{v'(r)}{r^2}\right)(x_i\p_{x_j}F^o).\]
Denote by $E=(E_{ij})=(x_i\p_{x_j}F^o)$ and $I=(\delta_{ij})$,
By using \eqref{mix3}, we get
\begin{eqnarray}\label{summ}
S_k[u]=\sum_{l=0}^k \binom{k}{l}\left(\frac{v'(r)}{r}\right)^{k-l} \left(\frac{v''(r)}{r}-\frac{v'(r)}{r^2}\right)^{l}S_k({\underbrace{I, \cdots, I}_{k-l}, \underbrace{E, \cdots, E}_{l}})
\end{eqnarray}
We claim that 
\begin{eqnarray}\label{claim}
S_k({\underbrace{I, \cdots, I}_{k-l}, \underbrace{E, \cdots, E}_{l}})=0, \hbox{ for }l\ge 2.
\end{eqnarray}
In fact, 
\begin{eqnarray*}
{ \frac{\binom{n}{l}}{\binom{n}{k}}}S_k({\underbrace{I, \cdots, I}_{k-l}, \underbrace{E, \cdots, E}_{l}})&=&S_{l}(E)=\frac{1}{l!}\sum_{\substack{i_1, \cdots i_{l}
\\ j_1, \cdots j_{l}}}\delta_{i_1\cdots i_l}^{j_1\cdots j_l}E_{i_1j_1}\cdots E_{i_lj_l}
\\&=&\frac{1}{l!}\sum_{\substack{i_1, \cdots i_{l}
\\ j_1, \cdots j_{l}}}\delta_{i_1\cdots i_l}^{j_1\cdots j_l}x_{i_1}x_{i_2}\cdots x_{i_l}\p_{x_{j_1}}F^o\p_{x_{j_2}}F^o\cdots \p_{x_{j_l}}F^o
\end{eqnarray*}
Since  $\delta_{i_1\cdots i_l}^{j_1\cdots j_l}$ is anti-symmetric with respect to the indices $i_1$ and $i_2$, while $$x_{i_1}x_{i_2}\cdots x_{i_l}\p_{x_{j_1}}F^o\p_{x_{j_2}}F^o\cdots \p_{x_{j_l}}F^o
$$ is symmetric with respect to the indices $i_1$ and $i_2$, we know the above summation is zero. That proves the claim \eqref{claim}.
It follows from \eqref{summ} and \eqref{claim} that
\begin{eqnarray}\label{summ1}
S_k[u]= \left(\frac{v'(r)}{r}\right)^{k}S_k(I, \cdots, I)+k \left(\frac{v'(r)}{r}\right)^{k-1}\left(\frac{v''(r)}{r}-\frac{v'(r)}{r^2}\right) S_k({\underbrace{I, \cdots, I}_{k-1}}, E).
\end{eqnarray}
Note that $$S_k(I, \cdots, I)=S_k(I)=\binom{n}{k},$$ and { from \eqref{mix2}, we have}
\begin{multline*}
S_k({\underbrace{I, \cdots, I}_{k-1}}, E)=\binom{n-1}{k-1}\frac1kS_1(E)=\binom{n-1}{k-1}\frac1k\sum_{i=1}^n (x_i\n_{i}F^o)=\binom{n-1}{k-1}\frac1kF^o(x)=\binom{n-1}{k-1}\frac r k.
\end{multline*}
The first assertion follows from \eqref{summ1}.

Note that $$FF_iu_j=\frac{v'(r)^2}{r}x_i\p_{x_j}F^o=\frac{v'(r)^2}{r}E_{ij},$$ using \eqref{claim}, we have
\begin{align*}
\sum_{i,j}S_k^{ij}[u]FF_iu_j&= k\frac{v'(r)^2}{r}S_k(\underbrace{A, \cdots, A}_{k-1}, E)
\\&=k\frac{v'(r)^2}{r}\left(\frac{v'(r)}{r}\right)^{k-1}S_k(\underbrace{I, \cdots, I}_{k-1}, E)
\\&=\binom{n-1}{k-1}r^{-(k-1)}v'(r)^{k+1}.
\end{align*}
We finish the proof of the second assertion.
\end{proof}

\

\section{Symmetrization with respect to mixed volumes}

\subsection{Mixed volumes}\

In this subsection, we review some basic concepts in the theory of convex bodies. An excellent book of the theory of convex bodies is by Schneider \cite{schn}.

Let $\mathcal{K}$ be the family of all convex bodies in $\rr^n$. A convex body is a compact, convex set with non-empty interior.

For two convex bodies $K, L\in \mathcal{K}$, the Minkowski sum of $K$ and $L$ is a new convex body given by
$$(1-t)K+tL:=\{(1-t)x+ty\in \rr^n: x\in K, y\in L\}, t\in [0,1].$$

Minkowski proved that the volume of $(1-t)K+tL$ is a polynomial in $t$, the coefficients of which are the so-called mixed volumes $W_k(K,L)$. Precisely,
$$\hbox{Vol}\left((1-t)K+tL\right)
= \sum_{k=0}^{n}\binom{n}{k}(1-t)^{n-k}t^kW_{k}(K,L).$$
Especially, $W_{0}(K,L)=\hbox{Vol}(K)$ and $W_{n}(K,L)=\hbox{Vol}(L)$.

For our purpose, we choose $L=\mathcal{W}$, the unit Wulff ball with respect to the norm $F$.
We denote
$$W_{k,F}(K):= W_k(K, \mathcal{W}).$$
$W_{1,F}(K)$ is the anisotropic perimeter of $K$.

In the case that $K$ has a $C^2$ boundary, 
one can interpret $W_{k,F}(K)$ in terms of the anisotropic curvature integrals (see e.g.\cite[par. 38, Eq. (13)]{BF}): for $k=1,\cdots, n$,
\begin{eqnarray}\label{KL}
W_{k,F}(K)=\frac{1}{n\binom{n-1}{k-1}}\int_{\p K}S_{k-1}(\k_F)  F(\nu)d\mathcal H^{n-1}.
\end{eqnarray}
For example, 
\[
W_{2,F}(K)= \int_{\de K} H_F F(\nu) d \mathcal H^{n-1}.
\]
One sees directly that
\begin{equation*}
\label{rad}
W_{k, F}(\mathcal{W}_{r})=\frac{1}{n}\frac{1}{r^{k-1}} P_{F}(\mathcal W_{r})=\kappa_{n}r^{n-k}.
\end{equation*}

In the case that $F$ is the Euclidean norm, $W_{k,F}(K)$ is the classical quermassintegral and can be interpreted in terms of the classical $k$-th mean curvature integrals.

\begin{defn} The anisotropic $k$-mean radius is defined by
\begin{eqnarray*}
\zeta_{k, F}(K):= \left( \frac{W_{k,F}(K)}{\kappa_n} \right)^{\frac{1}{n-k}}.
\end{eqnarray*}
\end{defn}

A basic property for the mixed volumes is that $W_{k,F}$ is monotone increasing with respect to the inclusion of
convex sets, namely, $$W_{k,F}(K_1)<W_{k,F}(K_2), \hbox{ if }K_1\subsetneq  K_2.$$ See e.g. \cite[Eq. (5.25)]{schn}.
As a direct consequence, 
\begin{eqnarray*}
\zeta_{k, F}(K_1)<\zeta_{k, F}(K_2), \hbox{ if }K_1\subsetneq  K_2.
\end{eqnarray*}

We recall the Alexandrov-Fenchel inequality, see e.g. \cite[p. 105]{X}, for an illustration of the following special case.
\begin{prop}[Aleksandrov-Fenchel inequalities]\

For $K\in \mathcal{K}$  which has $C^2$ boundary, it holds that
\begin{equation}
  \label{afineq}
\zeta_{k, F}(K) \ge 
 \zeta_{l, F}(K), \quad 0\le l < k
\le n-1,
\end{equation}
and the equality in \eqref{afineq} holds if and only if $K$ is
homothetic to $\mathcal W$.
\end{prop}
For the case $l=0$ and $k=1$, \eqref{afineq} is just the anisotropic isoperimetric inequality \cite{busemann}.

{ In the case that $K$ has a $C^2$ boundary, we have the following variational formula due to Reilly \cite[Theorem 3]{Re2}. See also He-Li \cite[Theorem 3.3]{HL} for a detailed proof. For $F=1$, the formula was first shown by Reilly in \cite[Theorem B]{Re1}. }
\begin{prop}[Reilly \cite{Re1, Re2}] Let $\S_t$ be a family of smooth $C^2$ closed, hypersurfaces evolving by the variational vector field $Y=\frac{\p \S_t}{\p t}$. Then
\begin{eqnarray}\label{var}
\frac{d}{d t}\int_{\S_t} S_k(\k_F) F(\nu)d\mathcal{H}^{n-1}=\int_{\S_t} (k+1)S_{k+1}(\k_F)\<Y, \nu\>d\mathcal{H}^{n-1}.
\end{eqnarray}
\end{prop}



\smallskip

\subsection{Symmetrization with respect to mixed volumes}\

Let $\Omega\subset \R^{n}$ be an open, bounded, convex set  with $C^{2}$ boundary. Define the following class of admissible functions
\[
\hspace{-1cm}\Phi_{0}(\Omega):= \left\{u\colon\Omega\to \R\, \big|\, u\in C^{{n}}(\Omega)\cap C(\bar \Omega),\, u = 0 \text{ on }\partial \Omega, u
\text{ is quasi-convex in }\Omega
 \right\}.
\]
Recall that $u$ is called quasi-convex, if all the {sub}level sets of $u$ are convex sets.
As in the last section, denote $m=\min u$ and $M=\max u$, 
\[
\O_t=\{x: u(x)<t\}, \quad \Sigma_t=\{x: u(x)=t\}, \hbox{ for }t\in[m, M].
\]
{If $u\in \Phi_{0}(\Omega)$, then it is nonpositive in $\Omega$. On the contrary, if $M=\max_{\Omega} u>0$, then by continuity $\{x\in \Omega\colon u(x)=M\}\Subset \Omega$ and $\{x\in \Omega\colon u(x)<M\}$ would be not convex.}

\begin{rem}

{The main point in order to use merely quasi-convex functions instead of strictly convex functions, as in Tso's paper \cite{Tso}, is given by the fact that this is the only hypothesis one really needs to apply the Alexandrov-Fenchel inequalities to its level sets, as we will do in Proposition \ref{equimeasprop} and Theorem \ref{pol}.}
\end{rem}

By Sard's theorem, $\{\S_t\}$ are non-degenerate $C^2$ hypersurfaces for a.e. $t\in [m, 0]$.


\begin{defn}
For $u\in \Phi_0(\O)$, the {\it anisotropic $k$-symmetrand} of $u$ is defined  
\[
\sym {k,F}: \overline{\Omegasym {k,F}}\to [m,{0}]
\]
 and
\begin{equation*}
	\sym {k,F}(x)=\sup\left\{t\leq 0: \zeta_{k,F}(\overline{\O_t})\le F^{o}(x)\right\}, 
\end{equation*}
where \Omegasym {k,F} is the Wulff ball 
having the same $k$-th mixed volumes or $k$-th anisotropic mean radius as $\Omega$, namely, $$ \zeta_{k,F}(\bar \O)=\zeta_{k,F}(\overline{\Omegasym {k,F}}).$$
\end{defn}

\begin{rem}In the case that $k=0$, $\sym {0,F}$ is the so-called convex symmetrand of $u$ (with respect to $F$) and $\Omegasym {0,F}$ is the Wulff ball with volume equal to $|\Omega|$ (see \cite{aflt}). This is the anisotropic counterpart of classical Schwarz symmetrization.  { We remark that for the convex-symmetrand $u^*_{0, F}$ of $u$ (with respect to $F$), it is not needed to assume quasi-convexity of $u$}. The case of $n=2$ and $k=1$ has been investigated by the first two authors \cite{dpgmaan}.
\end{rem}

	
We collect several basic properties of the anisotropic $k$-symmetrization in the following. We shall drop the subscript $k, F$ for simplicity.
\begin{prop}\label{basic-prop}
Denote  $\rho(r)=u^{*}(x)$, with $r=F^{o}(x)$. Denote $R=\zeta(\bar \O)$. We have
\begin{itemize}
\item[(i)] $\zeta(\overline{\O_t})$, as a function of $t\in [m, {0}]$, is an increasing function  and differentiable a.e. in $[m, {0}]$. And $\zeta(\overline{\O_m})=0$. 
\item[(ii)] $\rho$ is an increasing function and differentiable a.e. in $[0, R]$.
\item[(iii)]  $\rho(0)=m$. \[
\rho \textstyle\left(\zeta(\overline{\O_t})\right)=t, \hbox{ for a.e. }t\in [m,{0}].
\]
$\rho(r)={0}$ for $r\in [\zeta(\overline{\O_{{0}}}), R]$.
\item[(iv)] \[
\frac{d}{dt}\zeta(\overline{\O_t})=\left(\frac{d\rho}{dr}\Big|_{r=\zeta(\bar \O_t)}\right)^{-1}, \hbox{ for a.e. }t\in [m,{0}].
\]
\item[(v)]  The sub-level set $\Omega^*(t):=\{u^*(x) <t\}$, $t\in [m,{0}]$, are the Wulff ball having the same $k$-th mixed volume as $\Omega_{t}$, that is 
\begin{equation}
\label{qmr}
\zeta(\overline{\O^*_t}) = \zeta(\overline{\O_t}).
\end{equation}

\end{itemize}
\end{prop}
\begin{proof}
The proof is very similar to the one given in \cite{tr2} and \cite{Tso}. (i) follows from Sard's theorem and (ii) follows from (i). (iii) and (v) follows from the definition of $u^*$ and (iv) follows from (iii).
\end{proof}

Moreover, the following result holds.
\begin{prop} \label{reilly} {For a.e. $t\in [m, 0]$, we have}
\begin{equation*}
\frac{d}{dt} W_{k, F}(\overline{\O_t}) = \frac{1}{\binom{n}{k}}\int_{\Sigma_{t}} \frac{S_{k}(\kappa_{F}) F(\nu)}{F(\n u)}  d\mathcal H^{n-1}.
\end{equation*}
\begin{equation}
\label{reillyder}
\frac{d}{dt} \zeta_{k, F}(\overline{\O_t}) = \frac{1}{(n-k)\k_n\binom{n}{k}}\frac{1}{[\zeta_{k, F}(\overline{\O_t})]^{n-k-1}}\int_{\Sigma_{t}} \frac{S_{k}(\kappa_{F}) F(\nu)}{F(\n u)}  d\mathcal H^{n-1}.
\end{equation}
\end{prop}
\begin{proof} From \eqref{KL}, we see
\begin{eqnarray*}
W_{k, F}(\overline{\O_t})= \frac{1}{n\binom{n-1}{k-1}}\int_{\S_t}S_{k-1}(\k_F)  F(\nu)d\mathcal H^{n-1}
\end{eqnarray*}
Since the level sets $\{\S_t\}$ evolve by the vector field $$\frac{\p \S_t}{\p t}=\frac{1}{|\n u|}\nu=\frac{\n u}{|\n u|^2},$$
using \eqref{var}, we get
\begin{eqnarray*}
\frac{d}{dt}\int_{\S_t}S_{k-1}(\k_F)  F(\nu)d\mathcal H^{n-1}= \int_{\S_t}  \frac{kS_{k}(\k_F)}{|\n u|}d\mathcal H^{n-1}=\int_{\S_t} \frac{kS_{k}(\kappa_{F}) F(\nu)}{F(\n u)}  d\mathcal H^{n-1}.
\end{eqnarray*}
The assertion follows by direct computation.

\end{proof}

The following result states that the anisotropic $k$-symmetrization increases the $L^p$ norms of the function $u\in \Phi_0(\Omega)$.
\begin{prop}
\label{equimeasprop}
Let $u \in \Phi_0(\Omega)$. Then
\begin{equation*}
\|u\|_{L^{p}(\Omega)} \le \|u^{*}_{k,F}\|_{L^{p}(\Omega_{k,F}^{*})}, \quad 1 \le p < +\infty,
\end{equation*}
and
\begin{equation*}
\|u\|_{L^{\infty}(\Omega)} = \|u^{*}_{k,F}\|_{L^{\infty}(\Omega_{k,F}^{*})}.
\end{equation*}
\end{prop}
\begin{proof}
By definition, we have that $m=\min u=\min u^{*}_{k,F}$, and the second assertion follows. Moreover, by  \eqref{qmr} and the Aleksandrov-Fenchel inequality \eqref{afineq} we get
\begin{eqnarray*}
|\O_t|=W_{0,F}(\overline{\O_t}) &\le&   \kappa_{n}^{1-\frac{n}{n-k}} \left( W_{k,F}\left(\overline{\O_t}\right)\right)^{\frac{n}{n-k}} \\&= &\kappa_{n}^{1-\frac{n}{n-k}} \left( W_{k,F}\left(\overline{\O_t^*}\right)\right)^{\frac{n}{n-k}} 
\\&=& W_{0,F}(\overline{\O^*_t})=|\O^*_t|.
\end{eqnarray*}
It follows, by Fubini's theorem, that
{ \[
\int_{\Omega}|u|^{p}dx=p\int_{m}^{0} (-t)^{p-1}|\Omega_{t}|dt \le 
p\int_{m}^{0} (-t)^{p-1}|\Omega_{t}^{*}|dt =\int_{\Omega^{*}_{k,F}}|u^{*}_{k,F}|^{p}dx.
\]}
\end{proof}

\smallskip
\subsection{Anisotropic Hessian integral and P\'olya-Szeg\H o type inequalities}\
\begin{defn}
Let $u \in \Phi_0(\Omega)$. For $k=1, \cdots, n$, the anisotropic $k$-Hessian integral of $u$ is defined by
\begin{equation*}
I_{k,F}[u,\Omega]=\int_{\Omega}(-u)S_{k,F}[u] \, dx= \int_{\Omega}(-u)S_{k}(A_F[u]) \, dx.
\end{equation*}
where $A_F[u]$ is given by \eqref{A_F}.
\end{defn}
\begin{rem}By \eqref{byparts}, we see
\begin{eqnarray*}
I_{k, F}[u,\Omega]= \frac1k\int_{\O}  \sum_{i,j}S_k^{ij}[u]FF_iu_j \, dx.
\end{eqnarray*}
In particular, when $k=1$, 
\begin{eqnarray*}
I_{1,F}[u,\Omega]= \int_{\O}  F^2(\n u) \, dx,
\end{eqnarray*}
which is the anisotropic Dirichlet integral.
\end{rem}

\begin{prop}
Let $u(x)=v(r),$ $r=F^o(x)$,  be an anisotropic radial function defined on $\mathcal{W}_{r_0}$ such that $v'(0)=0$ and $v(r_0)=0$. Then 
\begin{equation}
\label{intrad}
I_{k,F}[u,\mathcal W_{r_0}]= \kappa_n \binom{n}{k} \int_0^{r_0}r^{n-k}  v'(r)^{k+1}\, dr.
\end{equation}
\end{prop}
\begin{proof}
Recall from Proposition \ref{skrad} that \begin{eqnarray*}
S_{k,F}[u]=\binom{n-1}{k-1}r^{-(n-1)} \left(\frac{r^{n-k}}{k}(v'(r))^{k}\right)'.
\end{eqnarray*}
Using co-area formula, we have
\begin{eqnarray}\label{rad-eq1}
I_{k,F}[u,\mathcal W_{r_0}]=\int_{0}^{r_0} (-v(r))\binom{n-1}{k-1}r^{-(n-1)} \left(\frac{r^{n-k}}{k}(v'(r))^{k}\right)' \int_{\p \mathcal{W}_r}\frac{1}{|\n F^o|}d\mathcal{H}^{n-1} dr.
\end{eqnarray}
Since $\<x, \n F^o\>= F^o$, we see
\begin{eqnarray*}
&&\int_{\p \mathcal{W}_r}\frac{1}{|\n F^o|}d\mathcal{H}^{n-1}=\int_{\p \mathcal{W}_r}\frac{\frac1r\<x,\n F^o\>}{|\n F^o|}d\mathcal{H}^{n-1}
=\frac1r\int_{\p \mathcal{W}_r}\<x,\nu\>d\mathcal{H}^{n-1}
=\frac nr|\mathcal{W}_r|=n\k_nr^{n-1}.
\end{eqnarray*}
Substituting the above into \eqref{rad-eq1},   using integration by parts, we get
\begin{eqnarray*}
I_{k,F}[u,\mathcal W_{r_0}]=n\k_n\int_{0}^{r_0} (-v(r))\binom{n-1}{k-1}\left(\frac{r^{n-k}}{k}(v'(r))^{k}\right)'  dr
=\kappa_n \binom{n}{k}\int_{0}^{r_0} r^{n-k}(v'(r))^{k+1} dr.
\end{eqnarray*}
\end{proof}


Now we are ready to prove the following P\'olya-Szeg\H o type inequality for anisotropic $k$-Hessian integral.
\begin{thm}
\label{pol}
Let $u \in \Phi_0(\Omega)$. Then
\[
I_{k,F}[u,\Omega]\ge I_{k,F}[u^{*}_{k-1,F},\Omega^{*}_{k-1,F}].
\]
Equality holds if and only if  $\O$ is a Wulff ball and $u$ is an anisotropic radial function.
\end{thm}
\begin{proof}
As usual, we drop the subscript $F$ for simplicity. By \eqref{byparts}, we have
\[
I_{k}[u,\Omega]= \frac{1}{k}\int_{m}^{{0}}  \int_{\S_t} S_{k-1}(\kappa_{F}) F^{k}(\nabla u) F(\nu)\, d\mathcal H^{n-1}dt.
\]
On the other hand, since $\O_t$ is convex for all $t$, by using the Alexandrov-Fenchel inequality \eqref{afineq}, \eqref{KL}, the  
H\"older inequality and \eqref{reillyder}, we get
\begin{eqnarray*}
&&\left(n\kappa_{n}\binom{n-1}{k-1}\right)^{k+1}[\zeta_{k-1}(\overline{\O_t})]^{(n-k)(k+1)}\\&\le& \left(n\kappa_{n}\binom{n-1}{k-1}\right)^{k+1}[\zeta_k(\overline{\O_t})]^{(n-k)(k+1)}\\ &=&
\left(\int_{\S_t}S_{k-1}(\kappa_{F}) F(\nu) d\mathcal H^{n-1} \right)^{k+1} 
\\&\le &  \left(\int_{\S_t} \frac{S_{k-1}(\kappa_{F})}{F(\nabla u)} F(\nu) d\mathcal H^{n-1}\right)^{k} \int_{\S_t} S_{k-1}(\kappa_{F}) F(\nabla u)^{k}F(\nu)d \mathcal H^{n-1}\\&=&
\left\{(n-k+1)\k_n\binom{n}{k-1}[\zeta_{k-1}(\overline{\O_t})]^{n-k} \frac{d}{dt}\zeta_{k-1}(\overline{\O_t}) \right\}^{k}\int_{\S_t} S_{k-1}(\kappa_{F}) F(\nabla u)^{k}F(\nu)d \mathcal H^{n-1}.
\end{eqnarray*}
It follows that
\begin{eqnarray}\label{xeq1}
\frac1k\int_{\S_t} S_{k-1}(\kappa_{F}) F(\nabla u)^{k}F(\nu)d \mathcal H^{n-1}\ge  \k_n\binom{n}{k}\frac{[\zeta_{k-1}(\overline{\O_t})]^{n-k}}{\left[\frac{d}{dt}\zeta_{k-1}(\overline{\O_t}) \right]^{k}}.
\end{eqnarray}
Hence,  we get
\begin{eqnarray*}
I_{k}[u,\Omega] \ge  \k_n\binom{n}{k}\int_m^{{0}} \frac{[\zeta_{k-1}(\overline{\O_t})]^{n-k}}{\left[\frac{d}{dt}\zeta_{k-1}(\overline{\O_t}) \right]^{k}} \, dt 
= \kappa_{n}  \binom{n}{k} \int_{0}^{R}r^{n-k}(\rho_{k-1}'(r))^{k+1}\, dr, 
\end{eqnarray*}
where $R= \zeta_{k-1}(\overline{\O})$.
In the last equality we have made the change of variables $r=\zeta_{k-1}(\overline{\O_t})$ and used Proposition \ref{basic-prop}. 

Finally, recalling that  $\rho_{k-1}(r)=u_{k-1}^*(x)$ and $\zeta_{k-1}(\overline{\O})=\zeta_{k-1}(\overline{\O_{k-1}^*})$, taking into account of \eqref{intrad}, we get the assertion.

If the equality holds, one sees from the above proof that $\S_t$ is Wulff shape and $F(\n u)$ is constant on $\S_t$. This implies that $\O$ is a Wulff ball and $u$ is an anisotropic radial function.\end{proof}

\

\section{Applications}

\subsection{A comparison result}\

In this subsection we use the symmetrization with respect to mixed volumes to prove a sharp comparison result. 

\begin{thm}
Let  $\Omega \subset \R^{n}$ be an open, bounded, convex set {with $C^{2}$ boundary} and $f\in L^1(\Omega)$ be a nonnegative function. {Let $$f_{0,F}^{*}(x)=-(-f)_{0, F}^*=\inf\{t\ge 0: |\overline{\{f>t\}}|\le \kappa_n F^o(x)^n\}.$$}
Let $u \in \Phi_0(\Omega)$ 
satisfy
\begin{equation}
\label{pbmain}
S_{k,F}[u]\le f(x), \hbox{a.e. in }\Omega
\end{equation}
 Then 
  \begin{equation}
  \label{confr}
  u^{*}_{k-1,F}(x)\ge v(x) \text{ in } \Omega_{k-1,F}^{*},
    \end{equation}
    where $v$ is the unique anisotropic radially symmetric solution of the following symmetrized problem:
    \begin{equation}
\label{pbsim}
\begin{cases}
S_{k,F}[v]=f_{0,F}^{*}(x) &\text{ in }\Omega_{k-1,F}^{*}\\
v=0 &\text{ on }\de \Omega_{k-1,F}^{*},
\end{cases}
\end{equation}

\end{thm}
 
\begin{proof}
 We integrate both sides of the equation in \eqref{pbmain} on the sub-level set $\O_t$. Using \eqref{xeq0}, {the divergence theorem, \eqref{k-curv}} and \eqref{xeq1},  we get
 \begin{eqnarray}
 \label{pezprin}
 \int_{\O_t}f(x)\,dx\ge\int_{\O_t}S_{k,F}[u]\, dx&=& \frac 1 k \int_{\S_t} S_{k-1}(\kappa_{F}) F^{k}(\nabla u) F(\nu)\, d\mathcal H^{n-1} \nonumber
 \\&\ge & \k_n\binom{n}{k}\frac{[\zeta_{k-1}(\overline{\O_t})]^{n-k}}{\left[\frac{d}{dt}\zeta_{k-1}(\overline{\O_t}) \right]^{k}}.
\end{eqnarray}
On the other hand, denoting by $f^{*}_{0}(r)=f^{*}_{0}(x)$, $r=F^{o}(x)$, using the Hardy-Littlewood inequality and the Aleksandrov-Fenchel inequality \eqref{afineq}, it holds that
\begin{eqnarray}
\label{pezf}
\int_{\O_t}f(x)\, dx&\le & \int_{0}^{|\O_t|}f^{*}_{0,F}\left(\left(\frac{r}{\kappa_{n}}\right)^{\frac{1}{n}}\right)\,dr \nonumber\\& \le&   \int_{0}^{\kappa_{n}\zeta_{k-1}(\overline{\O_t})^{n}}f^{*}_{0,F}\left(\left(\frac{r}{\kappa_{n}}\right)^{\frac{1}{n}}\right)\,dr\nonumber\\&= &n \kappa_{n}\int_{0}^{\zeta_{k-1}(\overline{\O_t})} f^{*}_{0,F}(s)s^{n-1}\,ds
\end{eqnarray}
Combining \eqref{pezprin} and \eqref{pezf} we obtain
\[
n \int_{0}^{\zeta_{k-1}(\overline{\O_t})} f^{*}_{0,F}(s)s^{n-1}\,ds\ge  \binom{n}{k}\frac{[\zeta_{k-1}(\overline{\O_t})]^{n-k}}{\left[\frac{d}{dt}\zeta_{k-1}(\overline{\O_t}) \right]^{k}}.
\]
Let $r=\zeta_{k-1}(\overline{\O_t})$, recalling Proposition \ref{basic-prop} (iv), we get
\begin{eqnarray}\label{xxeq1}
 \rho_{k-1}'(r)\le \left(\frac{n}{\binom{n}{k}}r^{-(n-k)} \int_{0}^r f^{*}_{0,F}(s)s^{n-1}\,ds\right)^{\frac1k}, \hbox{ for a.e. }t.
\end{eqnarray}

Fix $x \in \Omega_{k-1}^{*}$. Integrating \eqref{xxeq1} over $[\bar r, \zeta_{k-1}(\overline{\O})]$, noting that $\rho_{k-1}(\zeta_{k-1}(\overline{\O}))=0$, we have
\[
u^{*}_{k-1}(x)=\rho_{k-1}(F^o(x))\ge-\bigg(\frac{n}{\binom{n}{k}}\bigg)^{\frac 1 k} \int_{F^o(x)}^{\zeta_{k-1}(\overline{\O})}\left(r^{-(n-k)}\int_{0}^{r}f_{0,F}^{*}(s)s^{n-1}\,ds\right)^{\frac 1 k}\,dr.
\]
On the other hand, in view of Proposition \ref{skrad}, we are able to solve the solution to \eqref{pbsim} as
$$v(x)= -\bigg(\frac{n}{\binom{n}{k}}\bigg)^{\frac 1 k} \int_{F^o(x)}^{\zeta_{k-1}(\overline{\O_{k-1}^*})}\left(r^{-(n-k)}\int_{0}^{r}f_{0,F}^{*}(s)s^{n-1}\,ds\right)^{\frac 1 k}\,dr.$$
The assertion follows.
\end{proof}

\begin{rem}
{
The above comparison result cannot apply to solutions to general $k$-Hessian equations since it applies only to quasi convex functions and the quasi-convexity of such solutions is know only in some special case (for example in the Euclidean case, if $k=2$ and $n=3$). For further readings on this we refer, for example, to \cite{mx,lmx,sa12}. 
}
\end{rem}

\smallskip
\subsection{Sharp Sobolev type inequalities}\

Let $u\in \Phi_0(\O)$ and $p\ge 1$, $k=1,\ldots,n$. Consider integral functionals of the form
\begin{eqnarray*}
I_{k,p,F}[u,\O]=\int_{\O} S^{ij}_{k}[u]F^{p-k}F_iu_j\, dx.
\end{eqnarray*}
It is direct to see 
$$I_{k,k+1,F}=k\,I_{k,F},$$ 
the anisotropic $k$-Hessian integral and $$I_{1,p,F}=\int_{\O} F^p(\n u)\, dx,$$ the anisotropic $p$-Dirichlet integral.

By similar argument as in last section, we are able to prove the following P\'olya-Szeg\H o inequality.
\begin{thm}
\label{pol-2}
Let $u \in \Phi_0(\Omega)$ and $p\ge 1$. Then
\[
I_{k, p, F}[u,\Omega]\ge I_{k, p, F}[u^{*}_{k-1,F},\Omega^{*}_{k-1,F}].
\]
Equality holds if and only if  $\O$ is a Wulff ball and $u$ is an anisotropic radial function.
\end{thm}
Using this, we prove the following sharp Sobolev type inequality.
\begin{thm}
\label{pol-3}
Let be $u \in \Phi_0(\Omega)$, $k=1,\ldots,n$. Then
\begin{equation}
\label{poleq}
\|u\|_{L^q(\O)}^p\le C(n,k,p,F) I_{k,p,F}[u,\O],
\end{equation}
for $q= \frac{np}{n-k+1-p}$ and $p<n-k+1$, where 
\[
C(n,k,p,F)=\left( \displaystyle \frac{p-1}{n-k+1-p}\right)^{p-1}\left[ k \binom{n}{k}\right]^{-1} \\\times \left\{ \displaystyle \frac{\Gamma \left(\frac{np}{k-1+p} \right)}{\Gamma \left(\frac{n}{k-1+p} \right)\Gamma \left(1+\frac{n(p-1)}{k-1+p} \right)\kappa_{n}}\right\}^{\frac{k-1+p}{n}}
\]
is the optimal costant.
\end{thm}
\begin{proof}
By Proposition \ref{equimeasprop} and Theorem \ref{pol-2}, it is sufficient to prove \eqref{poleq} for anisotropic radially symmetric functions. Then it reduces to a one-dimensional problem, and we can argue exactly as in the Euclidean case (see \cite[pp. 216-217]{tr2}).
\end{proof}

\begin{rem}
In the case $p>n-k+1$ and $q\le \infty$, arguing as in \cite[p. 216]{tr2} it is possible to obtain that there exists a positive constant $C=C(n,k,p,F,\Omega)$ such that for $u \in \Phi_0(\Omega)$
\[
\|u\|_{L^{q}(\Omega)}^{p} \le C I_{k,p,F}[u,\O].
\]
Similarly, if $p=n-k+1$, then
\[
\|u\|^{p}_{L^{\Psi}(\Omega)} \le C I_{k, p, F}[u;\Omega]
\]
where $L^{\Psi}(\Omega)$ is the Orlicz space associated to the function $\Psi(t)=e^{|t|^{\frac{p}{p-1}}}-1$.
\end{rem}

\


\begin{thebibliography}{10}
\bibitem{aflt} A.Alvino, V. Ferone, P.-L. Lions, G. Trombetti. 
\textit{Convex symmetrization and applications}. 
Ann. Inst. H. Poincar\'e Anal. non lin\'eaire 
14:275-293, 1997.


\bibitem{BF} T. Bonnesen, W. Fenchel, Theorie der Konvexen K\"orper, Springer Press, 1934.

\bibitem{Bandle} 
C. Bandle, Isoperimetric inequalities and applications,
Monographs and Studies in Mathematics 7. Pitman Boston, 1980.

\bibitem{BC} C.~Bianchini and G.~Ciraolo, {\it Wulff shape characterizations in overdetermined anisotropic elliptic problems}, Comm. Partial Differential Equations, 43 (2018), 790-820.

\bibitem{bntpoin}
B.~Brandolini, C.~Nitsch, and C.~Trombetti.
{\it New isoperimetric estimates for solutions to Monge-Amp{\`e}re
  equations}.
{Annales de l'Institut Henri Poincar{\'e} (C) Analyse non
  lin{\'e}aire} 26(4):1265--1275, 2009.
 
 \bibitem{bt07}
B.~Brandolini and C.~Trombetti.
{\it A symmetrization result for Monge-Amp{\`e}re type equations}.
{Math. Nachr.}, 280(5-6):467--478, 2007.

\bibitem{bt07bis}
B.~Brandolini and C.~Trombetti.
{\it Comparison results for Hessian equations via symmetrization}.
{J. Eur. Math. Soc.}, 9:561--575, 2007.

\bibitem{busemann} H. Busemann,
{\it The isoperimetric problem for Minkowski area}.
{Amer. J. Math.} 71:743--762, 1949.

\bibitem{CS} A. Cianchi and P. Salani,
{\it Overdetermined anisotropic elliptic problems}. 
{Math. Ann.}, 345(4):
859-881, 2009.


\bibitem{CFV} M.~Cozzi, A.~Farina and E.~Valdinoci,
{\it Monotonicity formulae and classification results for singular, degenerate, anisotropic PDEs}, Adv. Math. 293 (2016) 343-381.

\bibitem{dpg2}
F.~{Della Pietra} and N.~Gavitone.
{\it Upper bounds for the eigenvalues of Hessian equations}.
{ Ann. Mat. Pura Appl.}, 193(3):923--938, 2014.

\bibitem{dpg3}
F.~{Della Pietra} and N.~Gavitone.
{\it Stability results for some fully nonlinear eigenvalue estimates}.
{ Comm. Contemporary Math.}, 16:1350039, 23 pages, 2014.

\bibitem{dpgmaan} F.~{Della Pietra} and N.~Gavitone.
{\it Symmetrization with respect to the anisotropic perimeter and applications},  Math. Ann. 363:953-971, 2015.

\bibitem{dgmana}  Della Pietra F.,  Gavitone N., \textit{Sharp bounds for the first eigenvalue and the torsional rigidity related to some anisotropic operators}, Math. Nachr. 287, 194-209 (2014).

\bibitem{ga09}
N.~Gavitone.
 {\it Isoperimetric estimates for eigenfunctions of Hessian operators}.
 {Ricerche di Matematica}, 58(2):163--183, 2009.

\bibitem{HL} Y. He, H. Li,
{\it  Stability of hypersurfaces with constant $(r+1)$-th anisotropic mean curvature.} Illinois J. Math. 52 (2008), no. 4, 1301-1314.

\bibitem{lmx} {P. Liu, X.N. Ma, L. Xu, A Brunn Minkowski inequality for the Hessian eigenvalue in three-dimensional convex domain, Adv. Math. 224 (2010) 1616-1633.}

\bibitem{mx} {X.N. Ma, L. Xu, The convexity of solution of a class of Hessian equation in bounded convex domain in $\R^{3}$,J. Funct. Anal. 255 (2008) 1713-1723.}

\bibitem{polyaszego}
G.~P\'{o}lya and G.~Szeg\H o, 
\emph{Isoperimetric inequalities in
  {M}athematical {P}hysics}, 
  Annals of Mathematics Studies, vol.~27, Princeton University Press, Princeton, N. J., 1951. 


%
%
%

\bibitem{Re1}  R. C. Reilly, {\it Variational properties of functions of the mean curvatures for hypersurfaces in space forms}. {J. Differential Geometry} 8:465-477, 1973.

\bibitem{Re2} R. C. Reilly, {\it The relative differential geometry of nonparametric hypersurfaces.} Duke Math. J. 43, 705-721, 1976.

\bibitem{Reilly}
R. Reilly, {\it On the Hessian of a function and curvatures of its graph}, {\em Michigan Math. J.}, 20:373--383, 1974.

\bibitem{sa12}
P.~Salani.
 {\it Convexity of solutions and Brunn-Minkowski inequalities for Hessian
  equations in {$\mathbb R^3$}}.
Adv. Math., 229(3):1924--1948, Feb. 2012.

\bibitem{Sch} H. A. Schwarz, Beweis des Satzes, dass die Kugel kleinere Oberfl\"ache besitzt, als jeder andere K\"orper gleichen Volumens, Nachrichten Königlichen Gesellschaft Wissenschaften G\"ottingen (1884), 1-13.

\bibitem{schn}
R. Schneider.
{\em {Convex bodies: the {B}runn-{M}inkowski theory}}.
 Cambridge University Press, Cambridge, 1993.

\bibitem{ta2} G. Talenti, \emph{Some Estimates Of Solutions to Monge - Amp\`ere Type Equations in Dimension Two}, Ann. Scuola Norm. Sup. Pisa Cl. Sci. (4) 8:183-230, 1981.

\bibitem{tr2} N.S. Trudinger, \emph{On new isoperimetric inequalities and symmetrization}, J. Reine Angew. Math., 488:203-220, 1997.


\bibitem{Tso} K. Tso, \emph{On Symmetrization and Hessian Equations}, J. Anal. Math., 52:94-106, 1989.


\bibitem{WX}
G. Wang and C. Xia,
\emph{A characterization of the Wulff shape by an overdetermined anisotropic PDE}.
{Arch. Ration. Mech. Anal.}, 199(1): 99--115, 2011.

\bibitem{WX2}
G. Wang and C. Xia,
\emph{An optimal anisotropic Poincare inequality for convex domains}, Pac. J. Math., 258 (2012) 305-326. 



\bibitem{X} C. Xia, {\em Inverse anisotropic mean curvature flow and a Minkowski type inequality}
Adv. Math. 315:102-129, 2017.
%
\end{thebibliography}
\end{document}